\documentclass[runningheads,final]{llncs}

\usepackage[T1]{fontenc}
\usepackage{comment}
\usepackage{cite}
\usepackage{wrapfig}

\spnewtheorem{assumption}[theorem]{Assumption}{\bfseries}{\itshape}

\usepackage{latexsym}
\usepackage{amsmath}
\usepackage{amssymb}
\usepackage{amsfonts}
\usepackage{mathtools}       
\usepackage{stmaryrd}        
\usepackage{bm}              

\providecommand{\liminf}{\operatornamewithlimits{liminf}} 
\DeclareMathOperator{\Tr}{Tr}     
\DeclareMathOperator{\Var}{Var}   
\DeclareMathOperator{\Cov}{Cov}   
\DeclareMathOperator{\Cond}{Cond} 
\DeclareMathOperator{\diag}{diag} 
\providecommand{\E}{\mathbb{E}} 
\providecommand{\ind}[1]{\mathbb{I}\{#1\}} 

\providecommand{\x}{\times}

\renewcommand{\geq}{\geqslant} 
\renewcommand{\leq}{\leqslant} 

\DeclarePairedDelimiterX{\inner}[2]{\langle}{\rangle}{#1, #2}

\usepackage{xcolor}

\usepackage{hyperref}
\usepackage{cleveref}

\crefname{dfn}{Definition}{Definition}
\crefname{prop}{Proposition}{Proposition}
\crefname{thm}{Theorem}{Theorem}
\crefname{cor}{Corollary}{Corollary}
\crefname{asp}{Assumption}{Assumption}

\usepackage{minted}

\usepackage{algorithm}
\usepackage{algorithmic}

\usepackage{hyperref}
\usepackage{color}
\usepackage{tikz}
\usetikzlibrary{arrows,calc,positioning,decorations.pathreplacing,backgrounds,fit}
\usepackage{subcaption}

\definecolor{cSelectFill}{RGB}{226,236,248}
\definecolor{cSelectLine}{RGB}{53,110,180}
\definecolor{cExact}{RGB}{185,78,72}
\definecolor{cAccent}{RGB}{36,95,74}
\definecolor{cMoment}{RGB}{214,138,39}
\definecolor{cGrayText}{RGB}{90,96,104}

\tikzset{
  >=latex,
  axis/.style={line width=0.85pt, ->},
  guide/.style={line width=0.6pt, draw=cGrayText!70, dashed},
  boundary/.style={line width=1.25pt, draw=cSelectLine},
  exactboundary/.style={line width=1.05pt, draw=cExact, dashed},
  region/.style={fill=cSelectFill},
  callout/.style={draw=cGrayText!80, rounded corners=2pt, fill=white, inner sep=3pt, align=left},
  vlabel/.style={font=\footnotesize, text=cGrayText},
  important/.style={font=\footnotesize\bfseries},
  samplept/.style={circle, fill=cGrayText!50, inner sep=0.65pt},
  meanarrow/.style={->, line width=1.0pt, draw=cMoment!90!black},
  ew/.style={->, line width=1.0pt, draw=cAccent},
}

\begin{document}
\title{
Beyond IGO-Flow: Toward Convergence Analysis of IGO in Continuous Spaces
}
\titlerunning{Toward Convergence Analysis of IGO in Continuous Spaces}
%
\author{Ryosuke Kimura\inst{1}\orcidID{0009-0006-7940-9934}\and \\
Youhei Akimoto\inst{1,2}\orcidID{0000-0003-2760-8123}}
\authorrunning{R. Kimura and Y. Akimoto}
%
\institute{University of Tsukuba, Tsukuba, Japan \and
RIKEN Center for Advanced Intelligence Project, Tokyo, Japan \\
\email{ryosuke.kimura@bbo.cs.tsukuba.ac.jp, akimoto@cs.tsukuba.ac.jp}
}
\maketitle              
%


\begin{abstract}

Information-Geometric Optimization (IGO) provides a unified framework for black-box optimization by interpreting the adaptation of a search distribution as a natural gradient update. 
Despite its conceptual importance, the convergence theory of IGO remains limited: most existing results concern continuous-time idealizations such as the IGO flow, rather than discrete-time updates with non-infinitesimal learning rates.
In this paper, we study discrete-time IGO in continuous spaces, formulated as natural gradient updates in the expectation-parameter coordinates of an exponential family.
In particular, we analyze IGO over the multivariate Gaussian family on strongly convex quadratic objective functions.
Our analysis covers a setting that simultaneously incorporates full covariance adaptation, a fixed positive learning rate, and quantile-based weights.
In this setting, we prove that the covariance matrix converges to the zero matrix.
We further show that the mean vector converges to the global optimum, provided that the condition number of the appropriately scaled covariance matrix is bounded at sufficiently frequent iterations.
These results advance the convergence theory of IGO and help bridge the gap between the mathematical theory of IGO and practical covariance-adaptive search methods such as CMA-ES.

\keywords{
Information-Geometric Optimization  \and 
Natural Gradient \and 
Evolution Strategies \and 
Convergence Analysis.
}
\end{abstract}

\section{Introduction}

Many real-world optimization problems are black-box in the sense that only function evaluations are available and gradient information cannot be accessed~\cite{Larson2019dfomethod}.
Evolution strategies (ESs), especially CMA-ES~\cite{hansen2001CMAES}, have demonstrated strong empirical performance across a wide range of black-box optimization (BBO) problems in continuous domains~\cite{2018largescalecma,nomura2021warmstartcma,maki2019control,fujii2018topology}.
However, despite its practical success, the theoretical understanding of these methods remains limited.

A natural way to study ESs theoretically is through Information-Geometric Optimization (IGO), which provides a unified view of a broad class of evolutionary algorithms~\cite{ollivier2017igo}. 
IGO formulates optimization as the evolution of a parametric search distribution updated along the natural gradient. 
This perspective is particularly appealing for rank-based stochastic search methods, including algorithms closely related to CMA-ES. 
Despite its conceptual importance, the convergence theory of IGO is still far from complete.
Establishing theoretical guarantees for IGO as a unifying framework for evolutionary computation is of fundamental importance for improving its reliability, promoting its use in real-world optimization problems, and advancing research in evolutionary computation.

In this paper, we study discrete-time IGO in continuous domains, formulated in the expectation-parameter coordinates of an exponential family.
We focus on IGO over the multivariate Gaussian family on strongly convex quadratic objective functions, while retaining three key features of practical covariance-adaptive search methods: full covariance adaptation, a fixed positive learning rate, and quantile-based weights. 
Although this setting captures essential aspects of practical covariance-adaptive search methods, it also makes the analysis substantially more difficult than in continuous-time idealizations or isotropic models.

Our main contributions are as follows. 
First, we prove that the covariance matrix generated by discrete-time IGO converges to the zero matrix under full covariance adaptation. 
Second, we show that the mean vector converges to the global optimum, provided that the condition number of the appropriately scaled covariance matrix is bounded at sufficiently frequent iterations.
The present analysis does not yet give a fully unconditional convergence guarantee for the mean vector. 
Nevertheless, it isolates the remaining obstacle in terms of the conditioning of the scaled covariance matrix and thereby provides a concrete route toward a complete convergence theory. 
In this sense, our results help narrow the gap between the mathematical theory of IGO and practical covariance-adaptive search methods such as CMA-ES.

The rest of this paper is organized as follows.
In \Cref{sec:related_work}, we outline related research and clarify the differences between prior studies and the present study.
In \Cref{sec:IGO}, we explain IGO and present the formulation studied in this paper. 
In \Cref{sec:convergence_analysis}, we present the main convergence results. 
Specifically, we analyze the convergence of the covariance matrix and the mean vector in \Cref{subsec:C_convergence} and  \Cref{subsec:m_convergence}, respectively.
Finally, we conclude the paper in \Cref{sec:conclusion}.

\section{Related Work}\label{sec:related_work}

Information-Geometric Optimization (IGO) is a unified framework for black-box optimization based on natural gradient updates of a search distribution~\cite{ollivier2017igo}. 
Existing analyses of IGO can be broadly divided into two lines: (1) analyses of discrete-time updates with a finite learning rate, and (2) analyses of idealized continuous-time trajectories obtained in the infinitesimal-step-size limit, such as the IGO flow. 

For discrete-time IGO, existing analyses mainly establish monotone improvement guarantees, including recent divergence-based refinements~\cite{akimoto2013objective,guilmeau2025klcondition}.
There are also convergence analyses with finite learning rates for related algorithms~\cite{akimoto2012ng,akimoto2022ode}, including an IGO-like natural gradient method with a simple objective transformation~\cite{akimoto2012ng} and an ODE-based analysis proving geometric convergence for a simplified isotropic ES with step-size adaptation~\cite{akimoto2022ode}. 
However, these results do not cover discrete-time IGO with both quantile-based weights and full covariance adaptation, which is the setting considered here.

Most convergence analyses of IGO have been developed for the continuous-time IGO flow~\cite{tobias2012igoflow,akimoto2012igoflow,beyer2014convergence}. 
While these studies are mathematically important, they analyze idealized trajectories rather than finite-step updates.
In contrast, we study a less idealized and more practically relevant setting: discrete-time IGO in continuous spaces with multivariate Gaussian distributions, full covariance adaptation, a finite learning rate, and quantile-based weights. 

\section{Information-Geometric Optimization (IGO)}
\label{sec:IGO}

In this section, we present the IGO formulation studied in this paper. 
We first summarize the central idea of IGO and then introduce the specific formulation used in our analysis.
In particular, we consider IGO on an exponential family of probability distributions, parameterized by expectation parameters, together with a quantile-based truncation weight.


\subsection{Overview}

Let $f:\mathcal{X}\to\mathbb{R}$ be an objective function on the search space $\mathcal{X}$, and consider the minimization of $f$.
IGO is based on a parametric family of probability distributions $\{P_\theta\}_{\theta\in\Theta}$ on $\mathcal{X}$.
Rather than optimizing $f$ directly over $\mathcal{X}$, IGO iteratively updates the distribution parameter $\theta$ by natural gradient ascent on the surrogate objective
\begin{align}
    J_t(\theta) := \E_{x \sim P_\theta}\bigl[W_t^f(x)\bigr].
\end{align}
Here, $W_t^f(x)$ is a rank-based weight defined with respect to the current distribution $P_t = P_{\theta_t}$.
Typically, $W_t^f(x)$ is defined using a nonincreasing function $w:[0,1]\to\mathbb{R}$ as follows:
\begin{align}\label{eq:general_def_weight}
    W_t^f(x) := w\Big(\Pr_{x' \sim P_t}\bigl(f(x') \le f(x)\bigr)\Big).
\end{align}
Thus, $W_t^f(x)$ depends only on the rank (quantile level) of $f(x)$ under $P_t$.

The natural gradient with respect to the Fisher metric is defined by
\begin{align}
    \widetilde{\nabla}_{\theta} J_t(\theta) = F(\theta)^{-1}\nabla_{\theta} J_t(\theta),
\end{align}
where $F(\theta)$ denotes the Fisher information matrix~\cite{amari1998natural}.
The IGO update is then given by the natural gradient ascent
\begin{align}
    \theta_{t+1} = \theta_t + \tau \widetilde{\nabla}_{\theta} J_t(\theta_t),
\end{align}
with learning rate $\tau \in (0,1]$.
Throughout this paper, we analyze the infinite-population model, i.e., the deterministic limit in which the expectations in the update are evaluated exactly rather than approximated from finite samples.

\subsection{Quantile-based Truncation Weight}

Fix a quantile level $q \in (0,1)$, and let $\kappa_t^q$ denote the $q$-quantile of $f(x)$ under $P_t$, defined by
\begin{equation}
    \kappa_t^q =
    \sup\Bigl\{
        u \in \mathbb{R}:
        \Pr_{x\sim P_t}\!\bigl(f(x)\leq u\bigr)\ge q
        \; \text{and} \;
        \Pr_{x\sim P_t}\!\bigl(f(x)\geq u\bigr)\le 1-q
    \Bigr\}.
\end{equation}
In this paper, we use the quantile-based truncation weight
\begin{align}\label{eq:def_weight}
    W_t^f(x) := \frac{1}{q}\,\mathbb{I}\{f(x)\le \kappa_t^q\}.
\end{align}
For simplicity, we write $I_t^{\le}(x) := \mathbb{I}\{f(x) \leq \kappa_t^q\}$.
This weight corresponds to the case where $w(z) = q^{-1}\mathbb{I}\{z \leq q\}$ in \eqref{eq:general_def_weight}.

When $f(x)$ has no atom at the threshold $\kappa_t^q$ under $P_t$, we have
\begin{align}\label{eq:no_ties}
    \Pr_{x \sim P_t}\bigl(f(x)=\kappa_t^q\bigr) = 0,
    \qquad
    \Pr_{x \sim P_t}\bigl(f(x)\le \kappa_t^q\bigr)=q,
\end{align}
and therefore
\begin{align}\label{eq:normalized_weight}
    \E_{x \sim P_t}\bigl[W_t^f(x)\bigr] = 1.
\end{align}
In \Cref{subsec:formulation_IGO}, we formulate the IGO assuming normalized weights, as in \eqref{eq:normalized_weight}.
This assumption is satisfied in the analytical setting used in this study.

\subsection{IGO in Expectation Parameterization}
\label{subsec:formulation_IGO}

Assume that $\{P_\theta\}_{\theta\in\Theta}$ is an exponential family with respect to a reference measure $\nu$ on $\mathcal{X}$, with density of the form
\begin{align}
    p_\theta(x) = \exp\bigl(\beta^\top T(x)-\psi(\beta)\bigr)h(x),
\end{align}
with natural parameter $\beta$, sufficient statistic $T:\mathcal{X}\to\mathbb{R}^n$, and log-partition function
\begin{align}
    \psi(\beta) = \log \int \exp\bigl(\beta^\top T(x)\bigr) h(x)\nu(dx).
\end{align}
We use the expectation parameter
\begin{align}
    \eta_t = \E_{x \sim P_t}[T(x)] = \nabla_\beta\psi(\beta).
\end{align}

For exponential families, the natural gradient of $J_t$ with respect to the expectation parameter $\eta$ has the simple form~\cite{ollivier2017igo}
\begin{align}
    \widetilde{\nabla}_{\eta} J_t(\theta_t)
    = \nabla_{\beta} J_t(\theta_t)
    = \E_{x \sim P_t}\big[W_t^f(x)(T(x)-\eta_t)\big].
\end{align}
Hence, with learning rate $\tau\in(0,1]$, the parameter update reads
\begin{subequations}
\begin{align}
    \eta_{t+1}
    &= \eta_t + \tau \widetilde{\nabla}_{\eta} J_t(\theta_t) \\
    &= (1-\tau)\eta_t + \tau \E_{x \sim P_t}\big[W_t^f(x)T(x)\big].
    \label{eq:igo_eta_update}
\end{align}
\end{subequations}

In the Gaussian case $P_t = \mathcal{N}(m_t,C_t)$, define the weighted mean vector $m_t^\star$ and the weighted covariance matrix $C_t^\star$ by
\begin{align}\label{eq:weighted_param}
    m_t^\star := \E_{x \sim P_t}\big[W_t^f(x)x\big]
    ,\enspace
    C_t^\star := \E_{x \sim P_t}\big[W_t^f(x)(x-m_t^\star)(x-m_t^\star)^\top\big].
\end{align}
Then the parameter update~\eqref{eq:igo_eta_update} reduces to
\begin{align}\label{eq:igo_ml_update_gaussian}
    \begin{cases}
        m_{t+1} = (1-\tau)m_t + \tau m_t^\star, \\
        C_{t+1} = (1-\tau)C_t + \tau C_t^\star
        + \tau(1-\tau)(m_t^\star-m_t)(m_t^\star-m_t)^\top.
    \end{cases}
\end{align}
Hereafter, we assume that $P_t = \mathcal{N}(m_t, C_t)$ and focus on the update rule~\eqref{eq:igo_ml_update_gaussian}.

\section{Convergence Analysis of IGO in Continuous Spaces}
\label{sec:convergence_analysis}

We consider the convergence of the IGO update \eqref{eq:igo_ml_update_gaussian} on a $d$-dimensional strongly convex quadratic objective function $f:\mathbb{R}^d \to \mathbb{R}$ of the form
\begin{align}\label{eq:convex_quadratic_func}
    f(x) = x^\top A x \qquad (A \succ 0).
\end{align}
Since $P_t=\mathcal{N}(m_t,C_t)$ on $\mathbb{R}^d$ and $f$ is continuous, the random variable $f(x)$ under $x\sim P_t$ has a continuous distribution. 
In particular,
\begin{align}
    \Pr_{x\sim P_t}\bigl(f(x)=\kappa_t^q\bigr) = 0,
\end{align}
so \eqref{eq:no_ties} and \eqref{eq:normalized_weight} hold in this setting.

In \Cref{subsec:monotone}, we establish a bound on the one-step progress of the expected objective value, which in particular implies its monotone improvement.
This result is used in \Cref{subsec:C_convergence} to establish the convergence of the covariance matrix. 
In \Cref{subsec:m_convergence}, we combine the fact $\lim_{t \to \infty}C_t=O$ with Assumption~\ref{asup:cond_num_lemma_assumption} and prove by contradiction that the mean vector converges to the optimal solution, i.e., $\lim_{t \to \infty}m_t = 0$.

\subsection{Monotone Improvement}
\label{subsec:monotone}

Define $\alpha := (0, A)$.
The sufficient statistic for the Gaussian distribution is $T(x)=(x, xx^\top)$.
Hence, the objective can be written as
\begin{align}
    f(x) = \langle \alpha, T(x)\rangle = \langle A, xx^\top \rangle = x^\top A x.
\end{align}
Let $V(P_t) := \E_{x \sim P_t}[f(x)]$ denote the expected objective value under $P_t$.
Since $V(P_t) = \E_{x\sim P_t}[\langle \alpha, T(x) \rangle] = \langle \alpha,  \eta_t \rangle$,
premultiplying both sides of \eqref{eq:igo_eta_update} by $\alpha$ yields
\begin{align}
    \langle \alpha, \eta_{t+1} \rangle
    &= (1-\tau)\langle \alpha, \eta_t \rangle + \tau\big\langle \alpha, \E_{x \sim P_t}\big[W_{t}^f(x)T(x)\big] \big\rangle \\
    \Longrightarrow V(P_{t+1})
    &= (1-\tau)V(P_t) + \tau\E_{x\sim P_t}\big[W_t^f(x)f(x)\big]. \label{eq:eq1}
\end{align}
Therefore, since $W_t^f$ is defined by \eqref{eq:def_weight}, the evolution of $V(P_t)$ is determined by the average objective value over the selected $q$-quantile region.

\begin{proposition}\label{prop:final_transformation}
Consider the sequence $\{V(P_t)\}_{t \geq 0}$ generated by the IGO update~\eqref{eq:igo_ml_update_gaussian} applied to the strongly convex quadratic objective function~\eqref{eq:convex_quadratic_func}.
Define $\mu_{t,(i)} := \E_{x \sim P_t}\big[(f(x)-V(P_t))^i\big]$ for $i=2$ and $4$.
Then, for all $t \geq 0$, the following inequality holds:
\begin{align}\label{eq:eq_final}
    V(P_{t+1}) \leq V(P_t) - \tau(1-q)
    \sqrt{\tfrac{1}{\mu_{t,(4)} / \mu_{t,(2)}^2 + 3}}
    \sqrt{\mathstrut\mu_{t,(2)}}
\end{align}
That is, $V(P_t)$ is nonincreasing, i.e.,
$V(P_{t+1}) \leq V(P_t)$ for all $t \geq 0$.
\end{proposition}
\begin{proof}
We first show
\begin{align}\label{eq:inequality_V}
    V(P_{t+1}) \leq V(P_t) - \frac{\tau}{2q}\E_{x,y \overset{\mathrm{i.i.d.}}{\sim} P_t}\big[|I_t^\leq(x)-I_t^\leq(y)|\big]\sqrt{\tfrac{\mu_{t,(2)}^3}{\mu_{t,(4)}+3\mu_{t,(2)}^2}}.
\end{align}
Let $x, y \overset{\mathrm{i.i.d.}}{\sim} P_t$.
Since $W_t^f(x)=q^{-1}I_t^\leq(x)$, \eqref{eq:eq1} becomes
\begin{align}
    V(P_{t+1}) = (1-\tau)V(P_t) + \frac{\tau}{q}\E\big[I_t^\leq(x)f(x)\big].
    \label{eq:V_update_for_proof}
\end{align}
Since $P_t$ is Gaussian and $A \succ 0$, the scalar random variable
$f(x) = x^\top A x$ has a continuous distribution under $P_t$. 
Therefore, we have $\Pr\big(f(x) = f(y)\big)=0$,
and $I_t^\leq(x)$ is an almost surely nonincreasing function of $f(x)$. 
In particular, for $(P_t \otimes P_t)$-almost everywhere, 
\begin{align}
    \bigl(I_t^\leq(x)-I_t^\leq(y)\bigr)\bigl(-f(x)+f(y)\bigr) \geq 0.
\end{align}
Thus, the assumptions of the improved Chebyshev sum inequality \cite[Theorem~21]{akimoto2022ode} are satisfied by $I_t^\leq(x)$ and $-f(x)$.
Then, applying \cite[Theorem~21 with $K=2$]{akimoto2022ode}, we obtain
\begin{align}
    \MoveEqLeft[1]
    -\E\big[I_t^\leq(x)f(x)\big] \notag \\
    &\geq -\E\big[I_t^\leq(x)\big]\E\big[f(x)\big] 
    + \frac{1}{4}\E\big[|I_t^\leq(x)-I_t^\leq(y)|\big]
      \E\big[|f(x)-f(y)|\big].
\end{align}
Since $\E[I_t^\leq(x)] = q$ and $\E[f(x)] = V(P_t)$, this yields
\begin{align}
    \E\big[I_t^\leq(x)f(x)\big]
    \leq qV(P_t)
    - \frac{1}{4}\E\big[|I_t^\leq(x)-I_t^\leq(y)|\big]
      \E\big[|f(x)-f(y)|\big].
    \label{eq:If_upper_bound}
\end{align}

Moreover, since $f(x)$ is a quadratic form of a Gaussian random vector, it has a finite fourth moment. 
Hence \cite[Proposition~22]{akimoto2022ode} applies and gives
\begin{align}
    \E\big[|f(x)-f(y)|\big] \geq
    2\sqrt{\tfrac{\mu_{t,(2)}^3}{\mu_{t,(4)}+3\mu_{t,(2)}^2}}.
\end{align}
Substituting this into \eqref{eq:If_upper_bound}, we obtain
\begin{align}
    \E\big[I_t^\leq(x)f(x)\big]
    \leq qV(P_t)
    - \frac{1}{2}\E\big[|I_t^\leq(x)-I_t^\leq(y)|\big]
    \sqrt{\tfrac{\mu_{t,(2)}^3}{\mu_{t,(4)}+3\mu_{t,(2)}^2}}.
\end{align}
Finally, plugging this bound into \eqref{eq:V_update_for_proof} yields \eqref{eq:inequality_V}. 

Now we prove \eqref{eq:eq_final} by explicitly evaluating the term
\begin{align}
    \E_{x,y \overset{\mathrm{i.i.d.}}{\sim} P_t}\big[|I_t^\leq(x)-I_t^\leq(y)|\big]
\end{align}
appearing in \eqref{eq:inequality_V}.
This term is evaluated as follows.
Since $\Pr\big(f(x) = \kappa_t^q\big) = 0$, we have $I_t^\leq(x) \in \{0,1\}$ almost surely for $x \sim P_t$.
Hence, $|I_t^\leq(x)-I_t^\leq(y)| = \mathbb{I}\{I_t^\leq(x) \neq I_t^\leq(y)\}$.
Therefore, $\E\big[|I_t^\leq(x)-I_t^\leq(y)|\big] = \Pr\big(I_t^\leq(x)\neq I_t^\leq(y)\big)$.
Since $x, y \overset{\mathrm{i.i.d.}}{\sim} P_t$,
\begin{align}
    \Pr\big(I_t^\leq(x)=1,\ I_t^\leq(y)=0\big)
    &= \Pr\big(I_t^\leq(x)=1\big)\Pr\big(I_t^\leq(y)=0\big) = q(1-q).
\end{align}
Similarly, since $\Pr\big(I_t^\leq(x)=0,\ I_t^\leq(y)=1\big) = (1-q)q$, we obtain
\begin{align}\label{eq:expectation_indicator_norm}
    \E\big[|I_t^\leq(x)-I_t^\leq(y)|\big] = q(1-q)+(1-q)q = 2q(1 - q).
\end{align}
By substituting \eqref{eq:expectation_indicator_norm} into \eqref{eq:inequality_V}, we obtain \eqref{eq:eq_final}.
This completes the proof. \qed
\end{proof}

\subsection{Convergence of the Covariance Matrix}
\label{subsec:C_convergence}

Now we establish the convergence of the covariance matrix $C_t$.
\Cref{prop:variance_convergence} shows that, once $\mu_{t,(4)}$ is controlled by $\mu_{t,(2)}^2$, the variance of $f$ converges, and the convergence of the covariance matrix $C_t$ then follows immediately.

\begin{proposition}\label{prop:variance_convergence}
Consider IGO updates~\eqref{eq:igo_ml_update_gaussian} applied to the strongly convex quadratic objective function~\eqref{eq:convex_quadratic_func}.
There exists a constant $L \geq 1$ such that
\begin{align}\label{eq:moment_bound}
    \mu_{t,(4)} \leq L \mu_{t,(2)}^2 \qquad \text{for all $t \geq 0$} .
\end{align}
In fact, we may take $L = 15$.
Then, we have $\lim_{t\to\infty} \mu_{t,(2)} = 0$, implying
\begin{align}
    \lim_{t\to\infty} A C_t = O.
\end{align}
\end{proposition}

\begin{proof}
The existence of a constant $L = 15$ can be derived in a straightforward manner from classical results on Gaussian quadratic forms; details are given in Appendix~\ref{sec:appendix_A}.

First, we show that the second centered moment $\mu_{t,(2)}$ (i.e., the variance of $f$) converges to $0$.
By inequality~\eqref{eq:eq_final}, we obtain
\begin{align}\label{eq:descent_sqrt_mu2_prop}
    V(P_{t+1}) \leq V(P_t)-\tau(1-q)\sqrt{\tfrac{1}{L+3}}\sqrt{\mathstrut \mu_{t,(2)}}.
\end{align}
Let $\delta := \tau(1-q)(L+3)^{-1/2} \in (0, \tau)$.
Rearranging \eqref{eq:descent_sqrt_mu2_prop}, we obtain
\begin{align}
    \delta\sqrt{\mathstrut\mu_{t,(2)}} \le V(P_t)-V(P_{t+1}).
\end{align}
Summing from $t = 0$ to $n$ yields
\begin{align}
    \delta\sum_{t=0}^n \sqrt{\mathstrut\mu_{t,(2)}} \leq \sum_{t=0}^n \big(V(P_t)-V(P_{t+1})\big) = V(P_0) - V(P_{n+1}) \leq V(P_0).
\end{align}
Since $V(P_0) < \infty$, 
\begin{align}
    \sum_{t=0}^{\infty}\sqrt{\mathstrut\mu_{t,(2)}} \leq \frac{V(P_0)}{\delta}  < \infty.
\end{align}
Since $\mu_{t,(2)} \geq 0$ for all $t \geq 0$, it follows that $\lim_{t\to\infty}\sqrt{\mathstrut\mu_{t,(2)}} = \lim_{t\to\infty}\mu_{t,(2)} = 0$.

Second, we show that $A C_t \to O$.
Since
\begin{align}
    \mu_{t,(2)} = \Var_{x \sim P_t}[f(x)] = 2\Tr(AC_tAC_t) + 4m_t^\top AC_tA m_t,
\end{align}
we have $\mu_{t,(2)} \geq 2\Tr(AC_tAC_t)$.
Hence $\mu_{t,(2)} \to 0$ implies $\Tr(AC_tAC_t) \to 0$.
Noting that $\Tr(A C_t A C_t) \geq \lambda_{\max}^2(A C_t)$, we have $\lim_{t\to\infty} \lambda_{\max}^2(A C_t) = 0$, implying $\lim_{t\to\infty} A C_t = O$.
This completes the proof.\qed
\end{proof}


We note that the proof of $\mu_{t,(2)} \to 0$ does not use any Gaussian-specific property.
Under~\eqref{eq:no_ties}, for IGO updates~\eqref{eq:igo_ml_update_gaussian}, this part of the argument only requires the moment bound~\eqref{eq:moment_bound}.

\subsection{Convergence of the Mean Vector}
\label{subsec:m_convergence}

We prove the convergence of the mean vector to the optimal solution as follows.
First, \Cref{lem:level_set_m} shows that $m_t^\top A m_t$ has a limit
$V' \geq 0$.
Therefore, it suffices to exclude the case $V' > 0$.
We prove it by contradiction.
\Cref{lem:existence_rho}, together with Assumption~\ref{asup:cond_num_lemma_assumption}, implies that $\Tr(A C_t)$ must grow asymptotically.
This contradicts \Cref{prop:variance_convergence}, and the resulting contradiction argument is formalized in \Cref{prop:m_convergence}.

To establish the convergence of the mean vector, we need an additional control on the shape of the covariance matrix, summarized in Assumption~\ref{asup:cond_num_lemma_assumption}.
While the previous subsection showed that $C_t \to O$, this does not prevent the covariance from becoming extremely anisotropic during the collapse.
A trivial way to avoid this difficulty is to assume uniform boundedness of the condition number of the covariance matrix during the whole optimization process, which is, however, not valid.
Instead, we construct a more realistic assumption on the shape of the covariance matrix.

For this purpose, consider the change in the trace $\Tr(A C_t)$ at each iteration. 
From the IGO update~\eqref{eq:igo_ml_update_gaussian} for the covariance matrix, let $\theta=(m,C)\in\Theta$ and define 
\begin{align}
    \Delta_C(\theta) :=
    \tau C^\star + \tau(1-\tau)(m^\star-m)(m^\star-m)^\top - \tau C.
\end{align}
Then $\Delta_C(\theta) \succeq - \tau C$, and we obtain $\Tr(AC_{t+1}) = \Tr(A C_t) + \Tr(A\Delta_C(\theta_t))$,
which is lower-bounded by $(1 - \tau) \Tr(A C_t)$. 
Based on this observation, we define
\begin{equation}
    \gamma := \inf_{\theta \in \Theta} \left(1+\frac{\Tr(A \Delta_C(\theta))}{\Tr(A C)}\right).\label{eq:beta}
\end{equation}
It is guaranteed that $\gamma \geq 1 - \tau$. 
With this quantity, we pose the following assumption on the condition number $\Cond(A C)$. 

\begin{assumption}\label{asup:cond_num_lemma_assumption}
Let $\{C_t\}_{t\geq0}$ be the sequence of covariance matrices generated by the IGO updates~\eqref{eq:igo_ml_update_gaussian} applied to the strongly convex quadratic objective function~\eqref{eq:convex_quadratic_func}.
Assume that there exists $R < \infty$ such that
\begin{align}
    \liminf_{T \to \infty} \frac{1}{T} \sum_{t = 0}^{T-1}\ind{\Cond(A C_t) \leq R} > \frac{ - \log(\gamma)}{ \log(\rho) - \log(\gamma)},\label{eq:asm}
\end{align}
where $\rho > 1$ denotes the constant appearing in \Cref{lem:existence_rho}.
\end{assumption}
The assumption states that $\Cond(A C_t)$ need not be bounded during optimization, but is bounded by a constant $R$ with higher frequency than the right-hand side of \eqref{eq:asm} would suggest. 
This is a more reasonable assumption than the boundedness of the condition number over the optimization process, as the bound can be violated infinitely many times. 

\begin{figure}[t]
    \centering
    \includegraphics[width=\hsize,clip,trim=0 0 0 32]{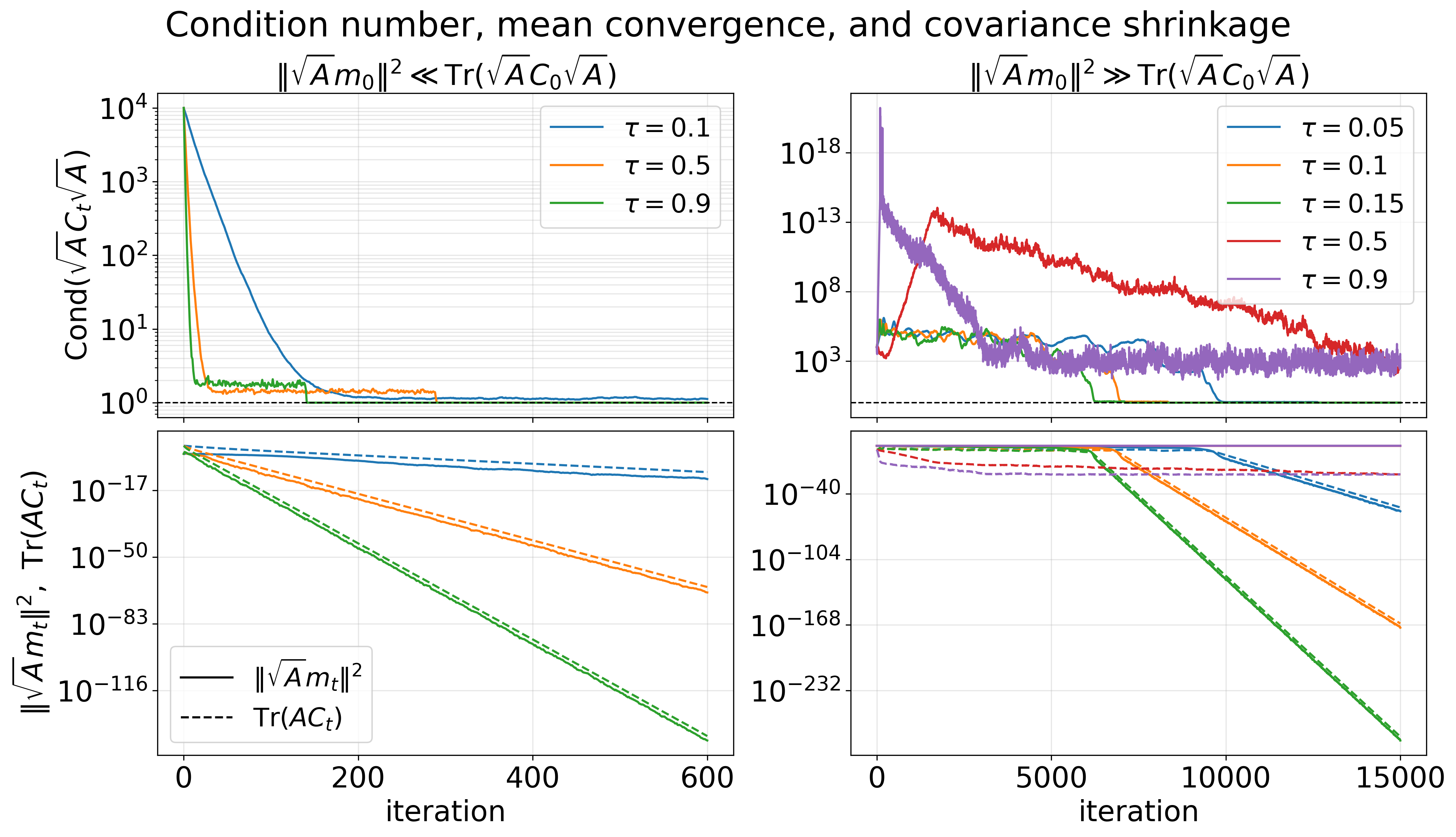}
    \caption{
    Empirical behavior of the condition number, mean convergence, and covariance shrinkage.
    Experiments are performed on the quadratic objective $f(x)=x^\top A x$ with $d=10$, $A=\diag(10^0,10^{4/9},\ldots,10^4)$, and the quantile level $q=0.3$.
    We evaluate two initializations where the ratio of $\|\sqrt{A}m_0\|^2$ to $\Tr(A C_0)$ is approximately $10^{-4}$ and $10^4$, respectively.
    }
    \label{fig:cond_num_experiment}
\end{figure}

To clarify the scope of Assumption~\ref{asup:cond_num_lemma_assumption},
we visualize the algorithmic behavior under two distinct initialization conditions in \Cref{fig:cond_num_experiment}.
Here, $\sqrt{A}$ is defined by $\sqrt{A}=Q\Lambda^{1/2}Q^\top$, where $A=Q\Lambda Q^\top$ is the eigendecomposition of $A$ and $\Lambda^{1/2}$ is the diagonal matrix of the square roots of the eigenvalues.
In the regime where $\|\sqrt{A}m_0\|^2 \ll \Tr(\sqrt{A}C_0\sqrt{A})$,
$\Cond(\sqrt{A}C_t\sqrt{A})$ rapidly approaches $1$,
demonstrating that the covariance matrix converges to a scalar multiple of $A^{-1}$.
Consequently, the assumption is empirically validated under this condition.
In contrast, when $\|\sqrt{A}m_0\|^2 \gg \Tr(\sqrt{A}C_0\sqrt{A})$,
serious anisotropy can emerge prior to mean convergence if a large learning rate is employed.
However, given a sufficiently small learning rate, the evolution of the shape of the covariance is better regulated, and the condition number eventually approaches approximately $1$.
Therefore, this assumption serves to exclude overly aggressive initial conditions for a given learning rate, where the collapse of the covariance is much faster and more anisotropic than mean convergence.

\begin{lemma}\label{lem:level_set_m}
Consider the sequences $\{P_t\}_{t \geq 0}$ generated by the IGO updates~\eqref{eq:igo_ml_update_gaussian} applied to the strongly convex quadratic objective function~\eqref{eq:convex_quadratic_func}.
There exists $V' \geq 0$ such that
$\lim_{t \to \infty} V(P_t) = \lim_{t \to \infty} m_t^\top A m_t = V'$.
\end{lemma}
\begin{proof}
By \Cref{prop:final_transformation}, the sequence $\{V(P_t)\}_{t \geq 0}$ is monotone nonincreasing.
Moreover, since $f(x) \geq 0$, we have $V(P_t) \geq 0$ for all $t \geq 0$, and hence the sequence is bounded below.
Therefore, $\{V(P_t)\}_{t \geq 0}$ converges to some $V' \geq 0$.

Recall that, for $P_t = \mathcal{N}(m_t, C_t)$ and $f(x) = x^\top A x$, $V(P_t) = m_t^\top A m_t + \Tr(AC_t)$.
By \Cref{prop:variance_convergence}, we have $AC_t \to O$, and hence $\lim_{t\to\infty}\Tr(AC_t) = 0$.
Therefore, $\lim_{t\to\infty}m_t^\top A m_t = \lim_{t\to\infty} V(P_t) - \lim_{t\to\infty}\Tr(A C_t) = V'$.
This completes the proof. \qed
\end{proof}

\begin{lemma}\label{lem:asymptotic_updates}
Consider the IGO updates~\eqref{eq:igo_ml_update_gaussian} applied to the strongly convex quadratic objective function~\eqref{eq:convex_quadratic_func} with the quantile level $q \in (0, 1/2)$. 
Define
\begin{align}
    \alpha_q := \Phi^{-1}(q),
    \qquad
    \lambda_q := \frac{\phi(\alpha_q)}{q},
\end{align}
where $\phi$ and $\Phi$ are the density and cumulative distribution functions of the standard normal distribution, respectively.
Let $V'$ be the limit given by \Cref{lem:level_set_m}, and assume that $V' > 0$.
For each $t$, let
\begin{align}\label{eq:u_and_v}
    u_t := \sqrt{A} C_t A m_t,
    \qquad
    v_t := \sqrt{C_t} A m_t.
\end{align}
Then, for $t$ sufficiently large, if $\Cond(A C_t) \leq R$ for some constant $R < \infty$, the IGO update~\eqref{eq:igo_ml_update_gaussian} for the covariance matrix in the $\sqrt{A}$-transformed coordinates satisfies
\begin{align}\label{eq:asymptotic_updates_formulas}
    \sqrt{A}C_{t+1}\sqrt{A}
    &= \sqrt{A}C_t\sqrt{A} 
    -\bigl(\tau\alpha_q\lambda_q+\tau^2\lambda_q^2\bigr)\frac{u_tu_t^\top}{\|v_t\|^2}
    + o\big(\lambda_{\max}(AC_t)\big).
\end{align}
\end{lemma}

\begin{proof}[Sketch]
We give a proof sketch focusing on the key geometric idea, while deferring the rigorous details to Appendix~\ref{sec:appendix_B}.

The proof has two conceptual steps.
First, we show that in standardized coordinates the quadratic truncation boundary becomes asymptotically flat, so the selection event is asymptotically equivalent to a Gaussian half-space.
Second, we evaluate the corresponding truncated Gaussian moments and plug them into the IGO updates~\eqref{eq:igo_ml_update_gaussian}.

Let $X_t = m_t + \sqrt{C_t} Z_t$, where $Z_t \sim \mathcal{N}(0,I)$, and define
\begin{align}
    B_t := \sqrt{C_t} A \sqrt{C_t}, \qquad
    e_t := \frac{v_t}{\|v_t\|}, \qquad
    \alpha_t := \frac{\kappa_t^q - m_t^\top A m_t}{2\|v_t\|}.
\end{align}
Then the truncation event can be rewritten as
\begin{align}
    f(X_t)\le \kappa_t^q
    \iff
    e_t^\top Z_t + \frac{Z_t^\top B_t Z_t}{2\|v_t\|} \le \alpha_t.
\end{align}
The key point is that the quadratic correction is asymptotically negligible. 
Indeed, writing $M_t:=\sqrt{A}m_t$ and $\Sigma_t:=\sqrt{A}C_t\sqrt{A}$, we have
$\|M_t\|^2\to V'>0$, $\Cond(\Sigma_t)\le R$, and $AC_t\to O$, which implies $\|B_t\|/\|v_t\| \to 0$.
Hence, on the Gaussian scale $Z_t=O_p(1)$, the curved boundary becomes asymptotically flat in standardized coordinates, so the selected region is asymptotically equivalent to a half-space orthogonal to $e_t$; see \Cref{fig:selection-linearization}.

More precisely, letting $Y_t:=e_t^\top Z_t\sim\mathcal N(0,1)$, we can write
\begin{align}
    S_t:=\{Y_t+R_t\le \alpha_t\},
    \qquad
    R_t:=\frac{Z_t^\top B_t Z_t}{2\|v_t\|}=o_p(1).
\end{align}
Since $\kappa_t^q$ is the $q$-quantile of $f(X_t)$, the selected mass is exactly $q$, which forces
\begin{align}
    \alpha_t \to \alpha_q := \Phi^{-1}(q) < 0.
\end{align}
Therefore, the true selection event is asymptotically equivalent to the Gaussian half-space $H_t := \{Y_t \leq \alpha_q\}$; see \Cref{fig:selection-quantile}.

To evaluate the truncated Gaussian moments, write
\begin{align}
    Z_t = Y_t e_t + U_t,
    \qquad U_t = (I-e_te_t^\top)Z_t.
\end{align}
Since the events $S_t$ and $H_t$ are asymptotically equivalent, the corresponding selected moments differ only by $o(1)$.
Moreover, under the half-space truncation $H_t=\{Y_t\le \alpha_q\}$, the above decomposition reduces the computation to the standard one-dimensional truncated Gaussian moments in the $e_t$-direction.
\begin{align}\label{eq:psi}
    \psi_t := \E\left[\frac{1}{q}\mathbb{I}\{S_t\} Z_t\right]
    = -\lambda_q e_t + o(1),
\end{align}
and
\begin{align}\label{eq:Psi}
    \Psi_t := \E\left[\frac{1}{q}\mathbb{I}\{S_t\} Z_t Z_t^\top\right]
    = I - \alpha_q\lambda_q\, e_t e_t^\top + o(1),
\end{align}
where $\lambda_q:=\phi(\alpha_q)/q$.

Finally, substituting these asymptotic moments into the weighted mean vector and the weighted covariance matrix~\eqref{eq:weighted_param} yields
\begin{align}
    m_t^\star = m_t + \sqrt{C_t}\psi_t,
    \qquad
    C_t^\star = \sqrt{C_t}(\Psi_t-\psi_t\psi_t^\top)\sqrt{C_t}.
\end{align}
Using
\begin{align}
    \sqrt{A}\sqrt{C_t}e_t = \frac{\sqrt{A}C_t A m_t}{\|v_t\|} = \frac{u_t}{\|v_t\|},
\end{align}
together with \eqref{eq:psi} and \eqref{eq:Psi}, we obtain the claimed expansion~\eqref{eq:asymptotic_updates_formulas}.
This completes the proof sketch. \qed

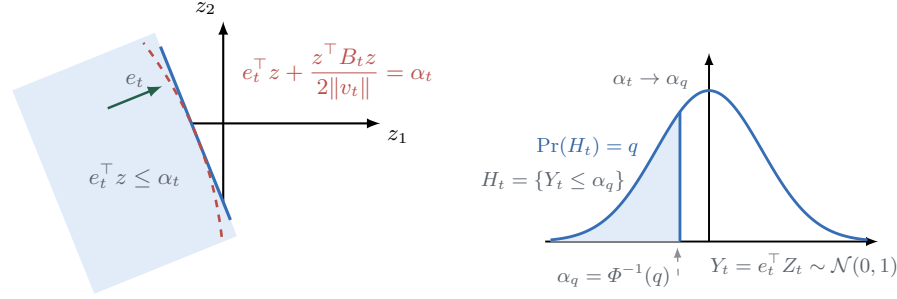
\begin{figure*}[t]
\centering

\begin{subfigure}[t]{0.49\textwidth}
\centering
\resizebox{\hsize}{!}{






\begin{tikzpicture}[x=0.92cm,y=0.7cm]
  \draw[axis] (-2.3,0) -- (2.4,0) node[below right=-1pt] {$z_1$};
  \draw[axis] (0,-1.95) -- (0,2.05) node[above left=-1pt] {$z_2$};

  \def\ang{22}      
  \def\a{-0.45}     
  \def\k{0.05}      

  \begin{scope}[rotate=\ang]
    \fill[region] (-2.8,-2.2) rectangle (\a,2.2);
    \draw[boundary] (\a,-1.85) -- (\a,1.9);

    \draw[exactboundary, domain=-2.10:2.10, samples=120]
      plot ({\a - \k*\x*\x},{\x});

    \draw[ew] (-1.55,1.15) -- (-0.65,1.15);
    \node[vlabel, above=1pt] at (-1.10,1.15) {$e_t$};

    \node[vlabel, text=cExact, align=left] at (1.92,0.10)
      {$e_t^\top z+\dfrac{z^\top B_t z}{2\|v_t\|}=\alpha_t$};
  \end{scope}
  
  \node[vlabel, align=center] at (-1.35,-1.05)
      {$e_t^\top z \le \alpha_t$};
\end{tikzpicture}
}
\caption{The quadratic truncation boundary is locally approximated by the half-space $e_t^\top z \leq \alpha_t$ in standardized coordinates.}
\label{fig:selection-linearization}
\end{subfigure}
\hfill
\begin{subfigure}[t]{0.49\textwidth}
\centering
\resizebox{\hsize}{!}{





\begin{tikzpicture}[x=0.82cm,y=5.85cm] 
  \draw[axis] (-3.1,0) -- (3.2,0);
  \draw[axis] (0,0) -- (0,0.50);

  \node[vlabel, below=2pt] at (1.8,0) {$Y_t=e_t^\top Z_t \sim \mathcal{N}(0,1)$};

  \path[fill=cSelectFill]
    plot[domain=-3.0:-0.55,samples=120] (\x,{0.4*exp(-\x*\x/2)})
    -- (-0.55,0) -- (-3.0,0) -- cycle;
  \draw[line width=1.25pt, draw=cSelectLine]
    plot[domain=-3.0:3.0,samples=160] (\x,{0.4*exp(-\x*\x/2)});

  \draw[boundary] (-0.55,0) -- (-0.55,{0.4*exp(-0.55*0.55/2)});
  
  \node[vlabel] (aq) at (-1.8,-0.09) {$\alpha_q=\Phi^{-1}(q)$};
  \draw[guide, ->] (aq.east) -- (-0.6,-0.015);

  \node[important, text=cSelectLine!90!black] at (-2.3,0.25) {$\Pr(H_t)=q$};
  \node[vlabel] at (-2.95,0.16) {$H_t=\{Y_t\le \alpha_q\}$};

  \node[vlabel, above=1pt] at (-1.1,0.38) {$\alpha_t\to \alpha_q$};
\end{tikzpicture}
}%
\caption{Because the selected mass is $q$, the threshold converges to $\alpha_q=\Phi^{-1}(q)$, and the selection event is asymptotically equivalent to the half-space $H_t=\{Y_t\le \alpha_q\}$.}
\label{fig:selection-quantile}
\end{subfigure}
\hfill

\caption{
Geometric interpretation of the proof. 
Left: Local linearization of the truncation boundary. 
Right: Convergence of the threshold to the Gaussian $q$-quantile.
}
\label{fig:proof-sketch}
\end{figure*}

\end{proof}

\begin{lemma}\label{lem:existence_rho}
Consider the IGO updates~\eqref{eq:igo_ml_update_gaussian} applied to the strongly convex quadratic objective function~\eqref{eq:convex_quadratic_func} with the quantile level $q \in (0, 1/2)$. 
Define
\begin{align}
    \tau_{\mathrm{crit}}(q):= -\frac{\alpha_q}{\lambda_q},
    \qquad
    r_t := \Tr(A C_t).
\end{align}

Suppose $\tau \in (0, \tau_{\mathrm{crit}}(q))$.
Let $V'$ be the limit given by \Cref{lem:level_set_m}, and assume that $V' > 0$.
Then, for sufficiently large $t$, if $\Cond(A C_t) \leq R$ for some constant $R < \infty$, there exists a constant $\rho > 1$
such that $r_{t+1} \geq \rho \, r_t$.
\end{lemma}
\begin{proof}
Let $\beta_q(\tau) := -(\tau\alpha_q\lambda_q+\tau^2\lambda_q^2)$.
Since $\tau < \tau_{\mathrm{crit}}(q) = -\alpha_q/\lambda_q$, we have $\beta_q(\tau) > 0$.
By \Cref{lem:asymptotic_updates}, we have
\begin{align}
    \sqrt{A}C_{t+1}\sqrt{A} = \sqrt{A}C_t\sqrt{A} + \beta_q(\tau)\,\frac{u_tu_t^\top}{\|v_t\|^2}
    + o\bigl(\lambda_{\max}(AC_t)\bigr).
\end{align}
Taking the trace yields
\begin{align}\label{eq:equation_r}
    r_{t+1} = r_t + \beta_q(\tau)\frac{\|u_t\|^2}{\|v_t\|^2} + o\bigl(\lambda_{\max}(A C_t)\bigr).
\end{align}
By the definitions of $u_t$ and $v_t$ in \eqref{eq:u_and_v}, we have
\begin{align}\label{eq:u_v_bound}
    \frac{\|u_t\|^2}{\|v_t\|^2}
    = \frac{m_t^\top A C_t A C_t A m_t}{m_t^\top A C_t A m_t}
    \geq \lambda_{\min}(A C_t).
\end{align}
Substituting this bound~\eqref{eq:u_v_bound} into \eqref{eq:equation_r}, we obtain
\begin{align}\label{eq:trace_lower_pre}
    r_{t+1} \geq r_t + \beta_q(\tau)\lambda_{\min}(AC_t)
    + o\big(\lambda_{\max}(AC_t)\big).
\end{align}
Since $\Cond(A C_t) \leq R$ and $r_t \leq d\,\lambda_{\max}(AC_t)$, we have 
\begin{align}
    \lambda_{\min}(A C_t) \geq \frac{1}{R}\lambda_{\max}(AC_t) \geq \frac{1}{Rd}r_t.
\end{align}
Moreover, since $\lambda_{\max}(A C_t) \leq r_t$, we have $o(\lambda_{\max}(A C_t)) = o(r_t)$.
Substituting this into \eqref{eq:trace_lower_pre}, we obtain
\begin{align}
    r_{t+1} \geq \left(1+\frac{\beta_q(\tau)}{Rd}\right)r_t
    + o(r_t).
\end{align}
Dividing both sides by $r_t > 0$, we get
\begin{align}
    \frac{r_{t+1}}{r_t}
    \geq
    1 + \frac{\beta_q(\tau)}{Rd}+o(1).
\end{align}
Hence, for all sufficiently large $t$, we can set $\rho := 1 + \frac{\beta_q(\tau)}{2Rd} > 1$.
This completes the proof. \qed
\end{proof}

\begin{proposition}\label{prop:m_convergence}
Consider the IGO updates~\eqref{eq:igo_ml_update_gaussian} applied to the strongly convex quadratic objective function~\eqref{eq:convex_quadratic_func} with the quantile level $q \in (0, 1/2)$.

Suppose that $\tau \in (0,\tau_{\mathrm{crit}}(q))$ and that
Assumption~\ref{asup:cond_num_lemma_assumption} holds.
Then
\begin{align}
    \lim_{t\to\infty}m_t = 0.
\end{align}
\end{proposition}

\begin{proof}
We prove the proposition by contradiction.
Assume, for contradiction, that $V' > 0$. 

Let $\bar{T}$ be sufficiently large. 
Then, for $t \geq \bar{T}$,
if $\Cond(A C_t) \leq R$, we have $\Tr(A C_{t+1}) \geq \rho \Tr(A C_{t})$ in light of \Cref{lem:existence_rho}. 
If not, we know that $\Tr(A C_{t+1}) \geq \gamma \Tr(A C_{t})$ from \eqref{eq:beta}. 
Therefore, for $t \geq \bar{T}$, we have
\begin{align}\label{eq:one_step}
    \Tr(A C_{t+1}) \geq \rho^{\ind{\Cond(AC_t) \leq R}} \gamma^{1 - \ind{\Cond(AC_t) \leq R}} \Tr(A C_t).
\end{align}
Taking logarithms in~\eqref{eq:one_step} and summing the resulting inequalities over $i = 0,\dots, T - 1$, we obtain
\begin{subequations}
\begin{align}
    \MoveEqLeft[1]
    \log \Tr(A C_{t + T}) - \log \Tr(A C_{t}) \notag \\
    &\textstyle= \sum_{i = 0}^{T - 1} \big( \log \Tr(A C_{t + i+1}) - \log \Tr(A C_{t + i}) \big)\\
    &\textstyle\geq \sum_{i = 0}^{T - 1} 
    \big(\ind{\Cond(A C_{t+i}) \leq R} \log(\rho) \notag \\
    &\textstyle\hspace{7em} + (1 - \ind{\Cond(A C_{t+i}) \leq R}) \log(\gamma) \big)\\
    &\textstyle= \sum_{i = 0}^{T - 1} \ind{\Cond(A C_{t+i}) \leq R} \big( \log(\rho) - \log(\gamma) \big) + T \log(\gamma).
\end{align}
\end{subequations}
Then Assumption~\ref{asup:cond_num_lemma_assumption} leads to
\begin{subequations}
\begin{align}
    \MoveEqLeft[0]
    \textstyle\liminf_{T \to \infty}\frac{1}{T} \big(
    \log \Tr(A C_{t + T}) - \log \Tr(A C_{t}) \big) \notag \\
    &\textstyle\geq \liminf_{T \to \infty}\frac{1}{T}\sum_{i = 0}^{T - 1} \ind{\Cond(A C_{t+i}) \leq R} \big( \log(\rho) - \log(\gamma) \big) + \log(\gamma) \\
    & > 0,
\end{align}
\end{subequations}
which implies $\lim_{T \to \infty} \Tr(A C_{t + T}) = +\infty$, contradicting $\lim_{t \to \infty} \Tr(A C_{t}) = 0$. 
Therefore, $V' = 0$. 
Hence, by \Cref{lem:level_set_m}, $\lim_{t\to \infty} m_t = 0$. 
This completes the proof.\qed
\end{proof}

Finally, \Cref{cor:result_summary} combines the convergence results established in
\Cref{subsec:C_convergence} and \Cref{subsec:m_convergence}.

\begin{corollary}\label{cor:result_summary}
Consider the IGO updates~\eqref{eq:igo_ml_update_gaussian} applied to the strongly convex quadratic objective function~\eqref{eq:convex_quadratic_func}.
\begin{enumerate}
    \item For any $q \in (0,1)$ and $\tau \in (0,1]$, we have $\lim_{t\to\infty}C_t = O$.
    \item Suppose, in addition, that Assumption~\ref{asup:cond_num_lemma_assumption} holds. 
    For any $q \in (0,1/2)$ and $\tau \in (0,\tau_{\mathrm{crit}}(q))$, with $\tau_{\mathrm{crit}}(q)$ defined in \Cref{lem:existence_rho}, we have $\lim_{t\to\infty}m_t = 0$.
    Consequently, $\lim_{t \to \infty}(m_t, C_t) = (0,O)$.
\end{enumerate}
\end{corollary}

\section{Conclusion}
\label{sec:conclusion}

In this paper, we studied discrete-time Information-Geometric Optimization in continuous spaces.
Focusing on IGO over the multivariate Gaussian family on strongly convex quadratic objective functions, we analyzed a practically relevant setting that combines full covariance adaptation, quantile-based weights, and natural-gradient updates in the expectation-parameter coordinates of an exponential family with a fixed positive learning rate.

In this setting, our analysis established the convergence of the covariance matrix to the zero matrix.
In addition, under Assumption~\ref{asup:cond_num_lemma_assumption}, we proved that the mean vector converges to the global optimum.
These results constitute a mathematically meaningful first step toward a complete convergence theory of IGO in continuous domains.

Our results also clarify the main remaining obstacle: controlling the evolving shape of the covariance matrix, rather than only its overall collapse.
In the current analysis, this issue is addressed through Assumption~\ref{asup:cond_num_lemma_assumption}, thereby highlighting the gap to a fully unconditional convergence guarantee.

An important direction for future work is therefore to remove this assumption and establish a fully unconditional convergence guarantee for the mean vector.
We hope that the present analysis provides a useful foundation for further progress in this direction.

\begin{credits}
\subsubsection{\ackname} 
This study was partly funded by JSPS KAKENHI 26K02993.

\subsubsection{\discintname}
The authors have no competing interests to declare that are relevant to the
content of this article.
\end{credits}



\begin{thebibliography}{10}
\providecommand{\url}[1]{\texttt{#1}}
\providecommand{\urlprefix}{URL }
\providecommand{\doi}[1]{https://doi.org/#1}

\bibitem{akimoto2012ng}
Akimoto, Y.: 
Analysis of a natural gradient algorithm on monotonic convex-quadratic-composite functions. 
In: Proceedings of the 14th Annual Conference on Genetic and Evolutionary Computation. 
pp. 1293--1300 (2012). 

\bibitem{akimoto2012igoflow}
Akimoto, Y., Auger, A., Hansen, N.: 
Convergence of the continuous time trajectories of isotropic evolution strategies on monotonic {$\mathcal{C}^2$}-composite functions. 
In: Parallel Problem Solving from Nature - PPSN XII. 
pp. 42--51 (2012). 

\bibitem{akimoto2022ode}
Akimoto, Y., Auger, A., Hansen, N.: An ode method to prove the geometric convergence of adaptive stochastic algorithms. Stochastic Processes and their Applications  \textbf{145},  269--307 (2022). 

\bibitem{akimoto2013objective}
Akimoto, Y., Ollivier, Y.: Objective improvement in information-geometric optimization. In: Foundations of Genetic Algorithms -- FOGA XII. pp. 1--10 (2013). 

\bibitem{amari1998natural}
Amari, S.i.: Natural gradient works efficiently in learning. 
Neural Computation  \textbf{10}(2), 251--276 (1998). 

\bibitem{beyer2014convergence}
Beyer, H.G.: 
Convergence analysis of evolutionary algorithms that are based on the paradigm of information geometry. 
Evolutionary Computation \textbf{22}(4), 679--709 (2014). 

\bibitem{fujii2018topology}
Fujii, G., Takahashi, M., Akimoto, Y.: Cma-es-based structural topology optimization using a level set boundary expression―application to optical and carpet cloaks. 
Computer Methods in Applied Mechanics and Engineering  \textbf{332},  624--643 (2018). 

\bibitem{tobias2012igoflow}
Glasmachers, T.: Convergence of the igo-flow of isotropic gaussian distributions on convex quadratic problems. 
In: Parallel Problem Solving from Nature - PPSN XII. 
pp. 1--10 (2012).

\bibitem{guilmeau2025klcondition}
Guilmeau, T., Chouzenoux, E., Elvira, V.: A divergence-based condition to ensure quantile improvement in black-box global optimization. 
IEEE Transactions on Evolutionary Computation  \textbf{29}(4),  1017--1028 (2025).

\bibitem{hansen2001CMAES}
Hansen, N., Ostermeier, A.: Completely derandomized self-adaptation in evolution strategies. 
Evolutionary Computation  \textbf{9}(2),  159--195 (2001). 

\bibitem{Larson2019dfomethod}
Larson, J., Menickelly, M., Wild, S.M.: Derivative-free optimization methods. 
Acta Numerica  \textbf{28},  287--404 (2019). 

\bibitem{maki2019control}
Maki, A., Sakamoto, N., Akimoto, Y., Nishikawa, H., Umeda, N.: 
Application of optimal control theory based on the evolution strategy (cma-es) to automatic berthing. 
Journal of Marine Science and Technology  \textbf{25}(1),  221--233 (2019). 

\bibitem{mathai1992quadratic}
Mathai, A.M., Provost, S.B.: 
Quadratic Forms in Random Variables: Theory and Applications, Statistics: Textbooks and Monographs, 
vol.~126. Marcel Dekker, (1992).

\bibitem{nomura2021warmstartcma}
Nomura, M., Watanabe, S., Akimoto, Y., Ozaki, Y., Onishi, M.: Warm starting cma-es for hyperparameter optimization. Proceedings of the AAAI Conference on Artificial Intelligence  \textbf{35}(10),  9188--9196 (2021).

\bibitem{ollivier2017igo}
Ollivier, Y., Arnold, L., Auger, A., Hansen, N.: Information-geometric optimization algorithms: A unifying picture via invariance principles. 
Journal of Machine Learning Research  \textbf{18}(18),  1--65 (2017).

\bibitem{2018largescalecma}
Varelas, K., Auger, A., Brockhoff, D., Hansen, N., ElHara, O.A., Semet, Y., Kassab, R., Barbaresco, F.: 
A comparative study of large-scale variants of cma-es. 
In: Parallel Problem Solving from Nature -- PPSN XV. 
pp. 3--15 (2018). 

\end{thebibliography}


\clearpage

\setcounter{section}{0}
\renewcommand{\thesection}{\Alph{section}}
\renewcommand{\theHsection}{appendix.\Alph{section}}

\appendix
\section{Proof of the existence of the constant $L$ in \Cref{prop:variance_convergence}}
\label{sec:appendix_A}

\begin{lemma}\label{lem:gaussian_quadratic_L15}
Let $X \sim \mathcal{N}(m,C)$, and let
\begin{align}
    Y = f(X) = X^\top A X,
\end{align}
where $A \in \mathbb{R}^{d\times d}$ is symmetric positive definite.
Let
\begin{align}
    \mu_{(2)} := \E\!\left[(Y-\E[Y])^2\right],
    \qquad
    \mu_{(4)} := \E\!\left[(Y-\E[Y])^4\right].
\end{align}
If $\mu_{(2)} > 0$, then
\begin{align}
    \mu_{(4)} \leq 15\mu_{(2)}^2.
\end{align}
\end{lemma}
\begin{proof}
By Theorem 3.3.2 in \cite{mathai1992quadratic}, the fourth cumulant of $Y$ is
\begin{align}
    c_4(Y) &= 48\Big(\Tr\big((AC)^4\big) + 4m^\top (AC)^3 A m\Big).
\end{align}
Since $\mu_{(4)} = 3\mu_{(2)}^2 + c_4(Y)$, it suffices to control $c_4(Y) / \mu_{(2)}^2$.

Let
\begin{align}
    M := \sqrt{A}m,
    \qquad
    \Sigma := \sqrt{A} C \sqrt{A}.
\end{align}
Because $AC$ and $\Sigma$ are similar, we may rewrite
\begin{align}
    \mu_{(2)} &= 2\Tr(\Sigma^2)+4M^\top \Sigma M, \\
    c_4(Y) &= 48\Big(\Tr(\Sigma^4)+4M^\top \Sigma^3 M\Big).
\end{align}

Set
\begin{align}
    S := \Tr(\Sigma^2),
    \qquad
    T := M^\top \Sigma M.
\end{align}
Since $\Sigma$ is symmetric positive semidefinite,
\begin{align}
    \Tr(\Sigma^4) \leq (\Tr(\Sigma^2))^2 = S^2,
\end{align}
and
\begin{subequations}
\begin{align}
    M^\top \Sigma^3 M 
    &= (\sqrt{\Sigma}M)^\top \Sigma^2 (\sqrt{\Sigma}M) \\
    &\leq \lambda_{\max}(\Sigma^2)\, M^\top \Sigma M \\
    &\leq \Tr(\Sigma^2)\, M^\top \Sigma M
    = ST.
\end{align}
\end{subequations}
Therefore,
\begin{align}
    c_4(Y) \leq 48(S^2 + 4ST).
\end{align}
Moreover,
\begin{align}
    \mu_{(2)} = 2S + 4T.
\end{align}
If $S = 0$, then $\Sigma = O$, hence $T=0$, and therefore $\mu_{(2)}=0$, contradicting the assumption.
Thus $S > 0$, and with $x := T/S \geq 0$,
\begin{align}
    \frac{c_4(Y)}{\mu_{(2)}^2}
    \leq 48\,\frac{S^2+4ST}{(2S+4T)^2}
    = 12\,\frac{1+4x}{(1+2x)^2}.
\end{align}
Now define
\begin{align}
    \phi(x) := \frac{1+4x}{(1+2x)^2},
    \qquad x \geq 0.
\end{align}
Then
\begin{align}
    \phi'(x) = -\frac{8x}{(1+2x)^3} \leq 0,
\end{align}
so $\phi$ is decreasing on $[0,\infty)$. 
Hence
\begin{align}
    \phi(x) \leq \phi(0) = 1,
\end{align}
which gives
\begin{align}
    \frac{c_4(Y)}{\mu_{(2)}^2} \leq 12.
\end{align}
Consequently,
\begin{align}
    \frac{\mu_{(4)}}{\mu_{(2)}^2} = 3 + \frac{c_4(Y)}{\mu_{(2)}^2} \leq 3 + 12 = 15.
\end{align}
This completes the proof. \qed
\end{proof}

\section{Complete proof of \Cref{lem:asymptotic_updates}}
\label{sec:appendix_B}

\begin{proof}
We divide the proof into four steps.

\noindent
\textbf{1. Linearization of the truncation region.}

First, we show that, under the assumption, the truncation region $\{f(x) \leq \kappa_t^q\}$ is asymptotically approximated, in the standardized coordinates, by a half-space determined by the linear term. 
Let $X_t \sim \mathcal{N}(m_t, C_t)$, and decompose it as
\begin{align}
    X_t = m_t + \sqrt{C_t} Z_t,\qquad Z_t\sim\mathcal N(0,I).
\end{align}
Substituting this decomposition into $f$, we obtain the representation
\begin{align}
    f(X_t) = m_t^\top A m_t + 2v_t^\top Z_t + Z_t^\top B_t Z_t,
\end{align}
where $B_t := \sqrt{C_t}A\sqrt{C_t}$.
Hence, the truncation condition $f(x) \leq \kappa_t^q$ can be rewritten as
\begin{align}
    f(X_t) \leq \kappa_t^q
    \iff \frac{1}{\|v_t\|}v_t^\top Z_t +\frac{Z_t^\top B_t Z_t}{2\|v_t\|} \leq \frac{\kappa_t^q-m_t^\top A m_t}{2\|v_t\|}.
\end{align}

Let
\begin{align}
    M_t := \sqrt{A}m_t,
    \qquad
    \Sigma_t := \sqrt{A}C_t\sqrt{A}.
\end{align}
Then, by \Cref{lem:level_set_m} and the assumption of this lemma,
\begin{align}
    \|M_t\|^2 = m_t^\top A m_t \to V' >0.
\end{align}
Hence, for sufficiently large $t$, there exists a constant $c_0 > 0$ such that
\begin{align}
    \|M_t\| \geq c_0.
\end{align}

Moreover, $B_t = \sqrt{C_t}A\sqrt{C_t}$ is similar to $AC_t$, and $\Sigma_t = \sqrt{A}C_t\sqrt{A}$ has the same eigenvalues as $AC_t$.
Therefore, we have
\begin{align}
    \|B_t\| = \lambda_{\max}(\Sigma_t),
    \qquad
    \Cond(\Sigma_t) = \Cond(AC_t) \leq R.
\end{align}

Now,
\begin{align}
    \|v_t\|^2
    = m_t^\top A C_t A m_t
    = M_t^\top \Sigma_t M_t
    \geq \lambda_{\min}(\Sigma_t)\|M_t\|^2.
\end{align}
Since $\Cond(\Sigma_t) \leq R$, we have
\begin{align}
    \lambda_{\min}(\Sigma_t)
    \geq \frac{\lambda_{\max}(\Sigma_t)}{R}.
\end{align}
Hence
\begin{align}
    \|v_t\|^2
    \geq \frac{c_0^2}{R}\lambda_{\max}(\Sigma_t)
    = \frac{c_0^2}{R}\|B_t\|.
\end{align}
Therefore,
\begin{align}
    \frac{\|B_t\|}{\|v_t\|}
    \leq \frac{\sqrt{R}}{c_0}\sqrt{\|B_t\|}.
\end{align}
By \Cref{prop:variance_convergence}, since $\lim_{t\to\infty}AC_t=O$, we have 
\begin{align}
    \|B_t\| = \lambda_{\max}(\Sigma_t) \to 0.
\end{align}
It follows that
\begin{align}
    \frac{\|B_t\|}{\|v_t\|} \to 0.
\end{align}

Hence the quadratic term
\begin{align}
    \frac{Z_t^\top B_t Z_t}{2\|v_t\|}
\end{align}
is negligible compared with the linear term on the $Z_t = O_p(1)$ scale. 
Indeed, we have
\begin{align}
    \left|\frac{Z_t^\top B_t Z_t}{2\|v_t\|}\right|
    \leq \frac{\|B_t\|}{2\|v_t\|}\|Z_t\|^2.
\end{align}
Since $\|B_t\|/\|v_t\| \to 0$ and $\|Z_t\|^2 = O_p(1)$, it follows that
\begin{align}
    \frac{Z_t^\top B_t Z_t}{2\|v_t\|} \to 0
    \qquad\text{in probability}.
\end{align}

Define
\begin{align}
    e_t := \frac{v_t}{\|v_t\|},
    \qquad
    \alpha_t:=\frac{\kappa_t^q-m_t^\top A m_t}{2\|v_t\|}.
\end{align}
Then
\begin{align}
    f(X_t)\le \kappa_t^q
    \iff
    e_t^\top Z_t + \frac{Z_t^\top B_t Z_t}{2\|v_t\|} \leq \alpha_t.
\end{align}
Hence, in standardized coordinates, the truncation region is asymptotically approximated by the half-space
\begin{align}
    e_t^\top z \leq \alpha_t.
\end{align}

\noindent
\textbf{2. Determination of the asymptotic half-space.}

Second, we show that the truncation region is asymptotically equivalent to the half-space determined by the $q$-quantile of the standard normal distribution.
Let
\begin{align}
    Y_t := e_t^\top Z_t \sim \mathcal N(0,1),
    \qquad
    R_t := \frac{Z_t^\top B_t Z_t}{2\|v_t\|}.
\end{align}
Then the truncation event can be written as
\begin{align}
    S_t := \{Y_t+R_t\le \alpha_t\}.
\end{align}
Since $\kappa_t^q$ is the $q$-quantile of $f(X_t)$, we have
\begin{align}
    q = \Pr(S_t) = \Pr(Y_t+R_t\le \alpha_t).
\end{align}

For any fixed $\varepsilon > 0$, we have
\begin{align}
    \MoveEqLeft
    \{Y_t \leq \alpha_t - \varepsilon\} \cap \{|R_t| \leq \varepsilon\} \notag \\
    &\subset
    \{Y_t + R_t \leq \alpha_t\}
    \subset
    \{Y_t \leq \alpha_t + \varepsilon\} \cup \{|R_t| > \varepsilon\}.
\end{align}
Therefore,
\begin{align}
    \Pr(Y_t \leq \alpha_t - \varepsilon) - \Pr(|R_t| > \varepsilon)
    \leq q \leq
    \Pr(Y_t \leq \alpha_t +\varepsilon) + \Pr(|R_t| > \varepsilon).
\end{align}
Since $Y_t \sim \mathcal{N}(0,1)$, it follows that
\begin{align}
    \Phi(\alpha_t - \varepsilon) - \Pr(|R_t| > \varepsilon)
    \leq q \leq
    \Phi(\alpha_t + \varepsilon) + \Pr(|R_t| > \varepsilon).
\end{align}
Moreover, by Step~1, $R_t \to 0$ in probability, so for every fixed $\varepsilon>0$,
\begin{align}
    \Pr(|R_t|>\varepsilon) \to 0.
\end{align}
Hence,
\begin{align}
    \Phi(\alpha_t - \varepsilon) - o(1)
    \leq q \leq
    \Phi(\alpha_t + \varepsilon) + o(1).
\end{align}
Therefore,
\begin{align}
    \alpha_t \to \Phi^{-1}(q)=:\alpha_q.
\end{align}

Finally, note that
\begin{align}
    \big\{\mathbb{I}\{S_t\} \neq \mathbb{I}\{Y_t \leq \alpha_q\}\big\}
    \subset
    \big\{|Y_t - \alpha_q| \leq |R_t| + |\alpha_t - \alpha_q|\big\}.
\end{align}
Hence, for any fixed $\varepsilon>0$,
\begin{align}
    \MoveEqLeft[2]
    \Pr\big(\mathbb{I}\{S_t\} \neq \mathbb{I}\{Y_t \leq \alpha_q\}\big) \notag \\
    &\leq \Pr\big(|Y_t - \alpha_q| \leq \varepsilon\big) + \Pr\big(|R_t| + |\alpha_t - \alpha_q| > \varepsilon\big).
\end{align}
Since $Y_t \sim \mathcal N(0,1)$ has a continuous density, while $R_t \to 0$ in probability and $\alpha_t \to \alpha_q$, the right-hand side converges to $0$. Therefore,
\begin{align}
    \Pr\big(\mathbb{I}\{S_t\} \neq \mathbb{I}\{Y_t \leq \alpha_q\}\big) \to 0.
\end{align}
Thus the truncation region is asymptotically equivalent to the half-space
\begin{align}
    e_t^\top z\le \alpha_q.
\end{align}

\noindent
\textbf{3. Truncated first and second moments.}

Third, we calculate the first and second truncated moments by replacing the truncation event with its asymptotic half-space.
Define
\begin{align}
    \psi_t := \E\left[\frac{1}{q}\mathbb{I}\{S_t\} Z_t\right],
    \qquad
    \Psi_t := \E\left[\frac{1}{q}\mathbb{I}\{S_t\} Z_tZ_t^\top\right],
\end{align}
and let
\begin{align}
    H_t := \{Y_t \leq \alpha_q\}.
\end{align}

By Step~2, we have $\Pr\big(\mathbb{I}\{S_t\} \neq \mathbb{I}\{H_t\}\big) \to 0$.
We first consider the truncated first moment. 
Since
\begin{align}
    \E\left[\frac{1}{q}\mathbb{I}\{S_t\}Z_t\right]
    - \E\left[\frac{1}{q}\mathbb{I}\{H_t\}Z_t\right]
    = \frac{1}{q}\E\left[\big(\mathbb{I}\{S_t\}-\mathbb{I}\{H_t\}\big)Z_t\right],
\end{align}
the Cauchy--Schwarz inequality yields
\begin{subequations}
\begin{align}
    \MoveEqLeft[2]
    \left\|
    \E\left[\frac{1}{q}\mathbb{I}\{S_t\}Z_t\right]
    - \E\left[\frac{1}{q}\mathbb{I}\{H_t\}Z_t\right]
    \right\| \notag \\
    &\qquad \leq \frac{1}{q}\E\left[\mathbb{I}\big\{\mathbb{I}\{S_t\}\neq \mathbb{I}\{H_t\}\big\}\|Z_t\|\right]\\
    &\qquad \leq \frac{1}{q}
    \big(\E[\|Z_t\|^2]\big)^{1/2}
    \Pr\big(\mathbb{I}\{S_t\} \neq \mathbb{I}\{H_t\}\big)^{1/2}.
\end{align}
\end{subequations}
Since $Z_t\sim\mathcal N(0,I)$, we have $\E[\|Z_t\|^2] < \infty$, and therefore
\begin{align}
    \E\left[\frac{1}{q}\mathbb{I}\{S_t\}Z_t\right]
    = \E\left[\frac{1}{q}\mathbb{I}\{H_t\}Z_t\right] + o(1).
\end{align}
Similarly, for the truncated second moment,
\begin{align}
    \E\left[\frac{1}{q}\mathbb{I}\{S_t\}Z_tZ_t^\top\right]
    - \E\left[\frac{1}{q}\mathbb{I}\{H_t\}Z_tZ_t^\top\right]
    = \frac{1}{q}\E\left[\big(\mathbb{I}\{S_t\}-\mathbb{I}\{H_t\}\big)Z_tZ_t^\top\right].
\end{align}
Hence,
\begin{subequations}
\begin{align}
    \MoveEqLeft[2]
    \left\|
    \E\left[\frac{1}{q}\mathbb{I}\{S_t\}Z_tZ_t^\top\right]
    - \E\left[\frac{1}{q}\mathbb{I}\{H_t\}Z_tZ_t^\top\right]
    \right\|_F \notag \\
    &\qquad \leq \frac{1}{q}\E\left[\mathbb{I}\big\{\mathbb{I}\{S_t\}\neq \mathbb{I}\{H_t\}\big\}\|Z_tZ_t^\top\|_F\right]\\
    &\qquad \leq \frac{1}{q}
    \big(\E[\|Z_tZ_t^\top\|_F^2]\big)^{1/2}
    \Pr\big(\mathbb{I}\{S_t\}\neq \mathbb{I}\{H_t\}\big)^{1/2}.
\end{align}
\end{subequations}
Since $\|Z_tZ_t^\top\|_F = \|Z_t\|^2$ and $Z_t \sim \mathcal N(0,I)$, we have $\E[\|Z_t\|^4] < \infty$. 
Hence,
\begin{align}
    \E\left[\frac{1}{q}\mathbb{I}\{S_t\}Z_tZ_t^\top\right]
    = \E\left[\frac{1}{q}\mathbb{I}\{H_t\}Z_tZ_t^\top\right] + o(1).
\end{align}

To compute the truncated moments under the asymptotic half-space, define
\begin{align}
    U_t := (I - e_te_t^\top)Z_t,
\end{align}
so that 
\begin{align}
    Z_t = Y_t e_t + U_t, 
    \qquad Y_t = e_t^\top Z_t.
\end{align}
Since $e_t$ is a deterministic unit vector and $Z_t \sim \mathcal{N}(0,I)$, both $Y_t$ and $U_t$ are linear transformations of the Gaussian vector $Z_t$. 
Hence $(Y_t, U_t)$ is jointly Gaussian.

Moreover, we have
\begin{align}
    \E[U_t] = (I-e_te_t^\top)\E[Z_t] = 0,
\end{align}
\begin{align}
    \Cov(U_t) 
    = (I-e_te_t^\top)\Cov(Z_t)(I-e_te_t^\top)^\top 
    = (I-e_te_t^\top)^2
    = I-e_te_t^\top.
\end{align}
Finally, the cross-covariance between $Y_t$ and $U_t$ is
\begin{subequations}
\begin{align}
    \Cov(Y_t,U_t)
    &= \E\big[Y_t U_t^\top\big] \\
    &= \E\big[e_t^\top Z_t\, Z_t^\top (I - e_te_t^\top)\big] \\
    &= e_t^\top \E[Z_tZ_t^\top](I - e_te_t^\top) \\
    &= e_t^\top (I - e_te_t^\top) = 0.
\end{align}
\end{subequations}
Since $(Y_t,U_t)$ is jointly Gaussian and uncorrelated, $Y_t$ and $U_t$ are independent.

Therefore,
\begin{subequations}
    \begin{align}
        \E\left[\frac{1}{q}\mathbb{I}\{H_t\} Z_t\right]
        &= \frac{1}{q}\E\left[\mathbb{I}\{Y_t \leq \alpha_q\}(Y_t e_t + U_t)\right] \\
        &= \frac{1}{q}\E\left[\mathbb{I}\{Y_t \leq \alpha_q\}Y_t\right]e_t
        + \frac{1}{q}\E\left[\mathbb{I}\{Y_t\le \alpha_q\}U_t\right].
    \end{align}
\end{subequations}
By independence of $Y_t$ and $U_t$, and since $\E[U_t]=0$,
\begin{align}
    \E\left[\mathbb{I}\{Y_t\le \alpha_q\}U_t\right]
    = \Pr(Y_t \leq \alpha_q)\E[U_t]
    =0.
\end{align}
Hence, by using the standard truncated normal identity
\begin{align}
    \E[Y_t | Y_t \leq \alpha_q] = - \frac{\phi(\alpha_q)}{\Phi(\alpha_q)},
\end{align}
we obtain 
\begin{align}
    \E\left[\frac{1}{q}\mathbb{I}\{H_t\} Z_t\right]
    = \E[Y_t\mid Y_t \leq \alpha_q]\,e_t
    = -\frac{\phi(\alpha_q)}{\Phi(\alpha_q)}\,e_t
    = -\lambda_q e_t,
\end{align}
where $\lambda_q := \phi(\alpha_q)/\Phi(\alpha_q)=\phi(\alpha_q)/q$.

Next,
\begin{align}
    \E\left[\frac{1}{q}\mathbb{I}\{H_t\} Z_tZ_t^\top\right]
    &= \frac{1}{q}\E\left[\mathbb{I}\{Y_t \leq \alpha_q\}(Y_t e_t + U_t)(Y_t e_t + U_t)^\top\right].
\end{align}
Expanding the product gives
\begin{subequations}
\begin{align}
    \MoveEqLeft[2]
    \E\left[\frac{1}{q}\mathbb{I}\{H_t\} Z_tZ_t^\top\right] \notag \\
    &= \frac{1}{q}\E\left[\mathbb{I}\{Y_t \leq  \alpha_q\}Y_t^2\right]e_te_t^\top 
    + \frac{1}{q}\E\left[\mathbb{I}\{Y_t \leq  \alpha_q\}Y_tU_t\right]e_t^\top \\
    &\hspace{2em} + \frac{1}{q}e_t\E\left[\mathbb{I}\{Y_t \leq  \alpha_q\}Y_tU_t^\top\right] 
    + \frac{1}{q}\E\left[\mathbb{I}\{Y_t \leq  \alpha_q\}U_tU_t^\top\right].
\end{align}
\end{subequations}
Again by independence of $Y_t$ and $U_t$, together with $\E[U_t]=0$, the two cross terms vanish.
Also,
\begin{align}
    \frac{1}{q}\E\left[\mathbb{I}\{Y_t\le \alpha_q\}U_tU_t^\top\right]
    = \E[U_tU_t^\top]
    = I - e_te_t^\top.
\end{align}
Therefore,
\begin{align}
    \E\left[\frac{1}{q}\mathbb{I}\{H_t\} Z_tZ_t^\top\right]
    &=
    \E[Y_t^2\mid Y_t\le \alpha_q]\,e_te_t^\top
    + (I - e_te_t^\top).
\end{align}
Hence, by using the standard truncated normal identity
\begin{align}
    \E[Y_t^2\mid Y_t\le \alpha_q] 
    = 1-\frac{\alpha_q\phi(\alpha_q)}{\Phi(\alpha_q)},
\end{align}
we obtain
\begin{subequations}
\begin{align}
    \E\left[\frac{1}{q}\mathbb{I}\{H_t\} Z_tZ_t^\top\right]
    &= \E[Y_t^2\mid Y_t\le \alpha_q]\,e_te_t^\top + (I - e_te_t^\top)\\
    &= \left(1-\frac{\alpha_q\phi(\alpha_q)}{\Phi(\alpha_q)}\right)e_te_t^\top + (I - e_te_t^\top)\\
    &= (1 - \alpha_q\lambda_q)e_te_t^\top + (I-e_te_t^\top) \\
    &= I - \alpha_q\lambda_q\,e_te_t^\top.
\end{align}
\end{subequations}

Consequently, we obtain
\begin{align}
    \psi_t = -\lambda_q e_t + o(1)
    \qquad\text{and}\qquad
    \Psi_t = I-\alpha_q\lambda_q\,e_t e_t^\top + o(1).\label{eq:psi_and_Psi}
\end{align}

\noindent
\textbf{4. Expansion of $m_{t+1}$ and $C_{t+1}$.}

Finally, we combine the asymptotic expansions of the truncated first and second moments obtained in Step~3 with the IGO update formulas, and derive the claimed expansions.

Using the IGO update formulas, we have
\begin{align}
    m_t^\star &= \E\left[ \frac{1}{q}\mathbb{I}\{S_t\} X_t \right] = m_t + \sqrt{C_t}\psi_t \\
    C_t^\star &= \E\left[ \frac{1}{q}\mathbb{I}\{S_t\} (X_t - m_t^\star)(X_t - m_t^\star)^\top \right] = \sqrt{C_t}(\Psi_t - \psi_t \psi_t^\top)\sqrt{C_t}.
\end{align}
By \eqref{eq:psi_and_Psi},
\begin{align}
    m_t^\star - m_t
    &= -\lambda_q \sqrt{C_t}e_t + o\big(\lambda_{\max}^{1/2}(C_t)\big),
\end{align}
where we used that $\sqrt{C_t}o(1)=o\big(\lambda_{\max}^{1/2}(C_t)\big)$.
Since 
\begin{align}
    e_t = \frac{v_t}{\|v_t\|}
    \qquad\text{and}\qquad
    \sqrt{A}\sqrt{C_t}v_t = \sqrt{A}C_t A m_t = u_t,
\end{align}
we have
\begin{align}
    \sqrt{A}\sqrt{C_t}e_t = \frac{u_t}{\|v_t\|}.
\end{align}
Therefore,
\begin{align}
    \sqrt{A}(m_t^\star - m_t)
    = -\lambda_q \frac{u_t}{\|v_t\|} + o\big(\lambda_{\max}^{1/2}(AC_t)\big).
\end{align}

Next, by \eqref{eq:psi_and_Psi},
\begin{align}
    \Psi_t - \psi_t\psi_t^\top
    &= I - (\alpha_q\lambda_q + \lambda_q^2)e_t e_t^\top + o(1),
\end{align}
and therefore
\begin{align}
    \MoveEqLeft[2]
    \sqrt{A}C_t^\star\sqrt{A} \notag \\
    &= \sqrt{A}C_t\sqrt{A}
    -(\alpha_q\lambda_q + \lambda_q^2)\sqrt{A}\sqrt{C_t}e_t e_t^\top \sqrt{C_t}\sqrt{A}
    + o\big(\lambda_{\max}(AC_t)\big).
\end{align}
Again using $\sqrt{A}\sqrt{C_t}e_t = u_t/\|v_t\|$,
\begin{align}
    \sqrt{A}C_t^\star\sqrt{A}
    &= \sqrt{A}C_t\sqrt{A}
    -(\alpha_q\lambda_q+\lambda_q^2)\,
    \frac{u_tu_t^\top}{\|v_t\|^2}
    + o\big(\lambda_{\max}(AC_t)\big).
\end{align}
Moreover, since
\begin{align}
    \left\|\frac{u_t}{\|v_t\|}\right\|
    = \frac{\|\sqrt{A}C_t A m_t\|}{\|\sqrt{C_t}A m_t\|}
    \leq \|\sqrt{A}\sqrt{C_t}\|
    = \lambda_{\max}^{1/2}(AC_t),
\end{align}
we have
\begin{align}
    \sqrt{A}(m_t^\star-m_t)(m_t^\star-m_t)^\top\sqrt{A}
    = \lambda_q^2\,\frac{u_tu_t^\top}{\|v_t\|^2} + o\big(\lambda_{\max}(AC_t)\big).
\end{align}
Therefore, substituting the above expansions into the IGO update~\eqref{eq:igo_ml_update_gaussian} for the covariance matrix in the $\sqrt{A}$-transformed coordinates yields
\begin{subequations}
\begin{align}
    \MoveEqLeft[0]
    \sqrt{A}C_{t+1}\sqrt{A} \notag \\
    &= \sqrt{A}C_t\sqrt{A}
    -\tau(\alpha_q\lambda_q+\lambda_q^2)\frac{u_tu_t^\top}{\|v_t\|^2}
    +\tau(1-\tau)\lambda_q^2\frac{u_tu_t^\top}{\|v_t\|^2}
    +o\big(\lambda_{\max}(AC_t)\big)\\
    &= \sqrt{A}C_t\sqrt{A}
    -(\tau\alpha_q\lambda_q+\tau^2\lambda_q^2)\,
    \frac{u_tu_t^\top}{\|v_t\|^2}
    +o\big(\lambda_{\max}(AC_t)\big).
\end{align}
\end{subequations}
This completes the proof.\qed
\end{proof}

\end{document}